\documentclass[11pt,a4paper]{article}
\usepackage{amssymb,amsmath,amsthm}
\usepackage[english]{babel}
\usepackage{t1enc}
\usepackage[latin2]{inputenc}
\usepackage{epsfig}
\usepackage{subfigure}
\usepackage{epstopdf}
\usepackage{lineno}
\usepackage{enumerate}

\usepackage{footnpag}

\usepackage{xspace} 

\newtheorem{thm}{Theorem}
\newtheorem{cor}[thm]{Corollary}
\newtheorem{definition}[thm]{Definition}
\newtheorem{lem}[thm]{Lemma}
\newtheorem{prop}[thm]{Proposition}

\newtheorem{obs}[thm]{Observation}

\def\R{\mathbb{R}}

\def\IIa{\mathcal{I}_{\mathbb R}}
\def\HH{\mathcal{H}}

\mathchardef\mhyphen="2D

\def\PQ{{\it Point-Quadrant}\xspace}
\def\QP{{\it Quadrant-Point}\xspace}
\def\PI{{\it Point-Interval}\xspace}
\def\IP{{\it Interval-Point}\xspace}
\def\IBI{{\it Interval-Bigger-Interval}\xspace}
\def\ISI{{\it Interval-Smaller-Interval}\xspace}
\def\ICI{{\it Interval-Crossing-Interval}\xspace} 
\def\PO{{\it Point-Octant}\xspace}
\def\OP{{\it Octant-Point}\xspace}

\def\dH{\mbox{\ensuremath{\mathcal D\mhyphen \mathcal H}}\xspace}
\def\dPQ{\mbox{\ensuremath{\mathcal D}}{\it -Point-Quadrant}\xspace}

\def\dPI{\mbox{\ensuremath{\mathcal D}}{\it -Point-Interval}\xspace}
\def\dIP{\mbox{\ensuremath{\mathcal D}}{\it -Interval-Point}\xspace}
\def\dIBI{\mbox{\ensuremath{\mathcal D}}{\it -Interval-Bigger-Interval}\xspace}   
\def\dISI{\mbox{\ensuremath{\mathcal D}}{\it -Interval-Smaller-Interval}\xspace}
\def\dICI{\mbox{\ensuremath{\mathcal D}}{\it -Interval-Crossing-Interval}\xspace}

\def\mPO{\mbox{\ensuremath{m(Point\mhyphen Octant)}}\xspace}
\def\mPQ{\mbox{\ensuremath{m(\mathcal D\mhyphen Point\mhyphen Quadrant)}}\xspace}
\def\mOP{\mbox{\ensuremath{m(Octant\mhyphen Point)}}\xspace}

\def\mPI{\mbox{\ensuremath{m(\mathcal D\mhyphen Point\mhyphen Interval)}}\xspace}
\def\mIP{\mbox{\ensuremath{m(\mathcal D\mhyphen Interval\mhyphen Point)}}\xspace}
\def\mIBI{\mbox{\ensuremath{m(\mathcal D\mhyphen Interval\mhyphen Bigger\mhyphen Interval)}}\xspace}
\def\mISI{\mbox{\ensuremath{m(\mathcal D\mhyphen Interval\mhyphen Smaller\mhyphen Interval)}}\xspace}
\def\mICI{\mbox{\ensuremath{m(\mathcal D\mhyphen Interval\mhyphen Crossing\mhyphen Interval)}}\xspace}

\makeatletter
\def\blfootnote{\gdef\@thefnmark{}\@footnotetext}
\makeatother

\begin{document}

\title{More on Decomposing Coverings by Octants}
\author{Bal\'azs Keszegh\thanks{Research supported by Hungarian National Science Fund (OTKA), under grant PD 108406 and by the J\'anos Bolyai Research Scholarship of the Hungarian Academy of Sciences.}\\
\and D\"om\"ot\"or P\'alv\"olgyi\thanks{Research supported by Hungarian National Science Fund (OTKA), under grant PD 104386 and by the J\'anos Bolyai Research Scholarship of the Hungarian Academy of Sciences.}
}


\date{}

\maketitle

\begin{abstract}
In this note we improve our upper bound given in \cite{KP} by showing that every $9$-fold covering of a point set in $\R^3$ by finitely many translates of an octant decomposes into two coverings, and our lower bound by a construction for a $4$-fold covering that does not decompose into two coverings. The same bounds also hold for coverings of points in $\R^2$ by finitely many homothets or translates of a triangle.
We also prove that certain dynamic interval coloring problems are equivalent to the above question.
\end{abstract}

\section{Introduction}
By an {\em octant}, in this paper we mean an open subset of $\R^3$ of the form $(-\infty,x)\times(-\infty,y)\times(-\infty,z)$ and the point $(x,y,z)$ is called the {\em apex} of the octant.
In \cite{KP} we have shown that every $12$-fold covering of a set in $\R^3$ by a finite number of octants decomposes into two coverings, i.e., if every point of some set $P$ is contained in at least $12$ members of a finite family of octants $\mathcal F$, then we can partition this family into two subfamilies, $\mathcal F=\mathcal F_1 \dot{\cup} \mathcal F_2$, such that every point of $P$ is contained in an octant from $\mathcal F_1$ and in an octant from $\mathcal F_2$.
We improve this constant in the following theorem, proved in Section~\ref{sec:upper}.

\begin{thm}\label{thm:9}
Every $9$-fold covering of a point set in $\R^3$ by finitely many octants decomposes into two coverings.
\end{thm}

The improvement comes from a careful modification of our original proof for $12$-fold covering, while the framework remains essentially the same.  
The equivalent dual version (see \cite{KP,PPT11}) of this statement is that any finite set of points in $\R^3$ can be colored with two colors such that any octant containing at least $9$ points contains both colors.
It was discovered in a series of papers by Cardinal et al.\ \cite{colorful,colorful2} and by us \cite{KPself} that this bound implies several further results for which earlier only doubly exponential bounds were known \cite{KPmulti}.
We denote by $m_{oct}$ the smallest integer such that every $m_{oct}$-fold covering of a finite point set in $\R^3$ by octants decomposes into two coverings, thus Theorem~\ref{thm:9} states that $m_{oct}\le 9$.
Using this new bound, the degrees of the polynomials in the following theorems have also been improved.
(A diagram describing the connection between different coloring problems can be found later in Figure~\ref{fig:diagram}.)

\begin{thm}[Keszegh-P\'alv\"olgyi \cite{KPself}]
For any positive integer $k$ and any given triangle $T$, any finite set of points can be colored with $k$ colors such that any homothet of $T$ containing at least $m_{oct}\cdot k^{\log (2m_{oct}-1)}$, thus $\Omega(k^{4.09})$ points, contains all $k$ colors. 
\end{thm}

\begin{thm}[Cardinal et al.\ \cite{colorful2}]
For any positive integer $k$, any $m_{oct}\cdot k^{\log (2m_{oct}-1)+1}$-fold covering of a subset of $R^3$ by finitely many octants can be decomposed into $k$ coverings.
\end{thm}

This theorem also has the following straight-forward corollaries.

\begin{cor}\
\begin{itemize}
\item For any positive integer $k$, any finite set of points in $\R^3$ can be colored with $k$ colors so that any octant with $\Omega(k^{5.09})$ points contains all $k$ colors. 
\item For any positive integer $k$, any $\Omega(k^{5.09})$-fold covering of a finite point set in $\R^2$ by homothets of a triangle decomposes into $k$ coverings.
\item For any positive integer $k$, any $\Omega(k^{5.09})$-fold covering of a finite point set in $\R^2$ by bottomless rectangles decomposes into $k$ coverings.
\end{itemize}
\end{cor}

Here a {\em bottomless rectangle} refers to a subset of the plane of the form $(x_1,x_2)\times (-\infty,y)$.
Note that it has been proved by Asinowski et al.\ \cite{A+13} that for any positive integer $k$, any finite set of points in $\R^2$ can be colored with $k$ colors such that any bottomless rectangle containing at least $3k-2$ points contains all $k$ colors (they also proved the lower bound that $3k-2$ cannot be changed to $1.67k$ in this statement).
A very general conjecture \cite[Problem 6.7 and after]{PPT11} implies that all the above parameters can also be replaced by $\Omega(k)$. 

We also give the following construction, which will be presented in Section~\ref{sec:lower}.

\begin{thm}\label{thm:lower}
For every triangle $T$ there is a finite point set $P$ such that for every two-coloring of $P$ there is a translate of $T$ that contains exactly $4$ points and all of these have the same color.
\end{thm}

This also implies $m_{oct}\ge 5$, as the intersection of octants with the plane $x+y+z=0$ give all homothets of the triangle $(-2,1,1),$ $(1,-2,1),$ $(1,1,-2)$, thus if we place the finite point set $P$ on this plane, then for any two-coloring of $P$ there will be an octant with exactly $4$ points, all of the same color.

We end the paper by discussing problems about coloring ordered intervals that turn out to be equivalent to the problem of decomposing octants, in Section~\ref{sec:int}.

\section{Proof of Theorem~\ref{thm:9}}\label{sec:upper}
The dual version of Theorem~\ref{thm:9} is that any finite set of points can be colored with two colors such that any octant containing at least $9$ points, contains both colors.
This is equivalent to the original problem.
We will prove the dynamic planar version of the dual problem, which is the following.
A {\em quadrant}  or {\em wedge} is a subset of the plane of the form $(-\infty,x)\times (-\infty,y)$.
We have to two-color a finite ordered planar point set $\{p_1,p_2,\ldots,p_n\}$ such that for every $t$ every quadrant that contains at least $9$ points from $P_t=\{p_1,\ldots,p_t\}$ contains both colors.
This dynamic planar version is also equivalent to the original problem, for the details, see \cite{KP,PPT11}.
(Briefly, the equivalence of the two problems is implied by the following containment-reversing bijections: an octant with apex $(x,y,z)$ is bijected to the point $(x,y)$ that ``appears'' at time $z$, while a point with coordinates $(a,b)$ is bijected to a mirrored quadrant with apex $(a,b)$, i.e., to the subset $(a,\infty)\times(b,\infty)$.)

A way to imagine this problem is that the points ``appear'' in order and at step $t$ we have to color the new point, $p_t$.
This is impossible to do in an online setting \cite{KNP}, i.e., without knowing in advance which points will come in which order.
Moreover, it was shown by Cardinal et al.\ \cite{colorful2} that such a coloring is even impossible in a so-called {\em semi-online} model, where points can be colored at any time after their arrival as long as every octant with $9$ (or any other constant number of) points contains both colors.
Our strategy, developed in \cite{KP}, builds a forest on the points such that any time any quadrant containing at least $9$ points contains two adjacent points from the same tree-component.
Therefore, after all the points arrived, any proper two-coloring of the forest 
will be such that any octant containing at least $9$ points contains both colors.


\begin{figure}
\begin{center}
\includegraphics[scale=0.7]{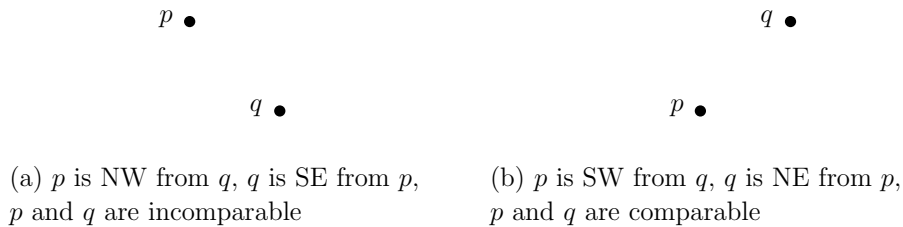}
\end{center}
\caption{Definition \ref{def:comparable}: the possible relations of two points.}
\label{fig:comparable}
\end{figure}

\begin{figure}
\begin{center}
\includegraphics[scale=0.6]{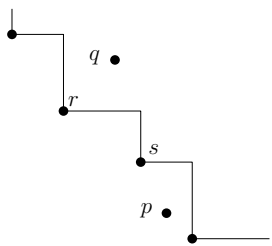}
\end{center}
\caption{Definition \ref{def:stair}: $p$ is below the staircase, $q$ is above the staircase, $r$ and $s$ are neighboring staircase points and $r$ is the left neighbor of $s$.}
\label{fig:stairdef}
\end{figure}
We start by introducing some notation (see also Figure \ref{fig:comparable}).

\begin{definition}\label{def:comparable}
We say that point $p=(p_x,p_y)$ is {\em northwest} (in short NW) from point $q=(q_x,q_y)$ if and only if $p_x<q_x$ and $p_y>q_y$.
In this case we also say that $q$ is {\em southeast} (in short SE) from $p$ and that $p$ and $q$ are {\em incomparable}.

Similarly, we say that point $p=(p_x,p_y)$ is {\em southwest} (in short SW) from point $q=(q_x,q_y)$ if and only if $p_x<q_x$ and $p_y<q_y$.
In this case we also say that $q$ is {\em northeast} (in short NE) from $p$ and that $p$ and $q$ are {\em comparable}.
\end{definition}

We can suppose that all points have different coordinates, as by a slight perturbation we can only get more subsets of the points contained in a quadrant (without losing others).

At any step $t$, we define a graph $G_t$ (which is actually a forest) on the points of $P_t$ and a vertex set $S_t$ of pairwise incomparable points called the {\em staircase}, recursively.
At the beginning $G_0$ is the empty graph and $S_0$ is the empty set.
A point on the staircase is called a {\em stair-point}.
Thus, before the $t^{th}$ step we have a graph $G_{t-1}$ on the points of $P_{t-1}$ and a set $S_{t-1}$ of pairwise incomparable points.
In the $t^{th}$ step we add $p_t$ to our point set obtaining $P_t$ and we will define the new staircase, $S_t$, and also the new graph, $G_t$, containing $G_{t-1}$ as a subgraph.
Before the exact definition of $S_t$ and $G_t$, we make some more definitions and fix some properties that will be maintained during the process (see Figure \ref{fig:stairdef}).

\begin{definition}\label{def:stair}
We say that a point $p$ of $P_t$ is {\em above} the staircase if there exists a stair-point $s\in S_t$ such that $p$ is NE from $s$.
If $p$ is not above or on the staircase, then we say that $p$ is {\em below} the staircase. A point below (resp.\ above) the staircase is called a {\em below-point} (resp.\ {\em above-point}).
At any time $t$, we say that two points of $S_t$ are {\em neighbors} if their $x$-coordinates are consecutive among the $x$-coordinates of the stair-points.
(Note that this does not mean that they are connected in the graph.)
We also say that $p$ is the {\em left} (resp.\ {\em right}) neighbor of $q$ if $p$ and $q$ are neighbors and the $x$-coordinate of $p$ is less (resp.\ more) than the $x$-coordinate of $q$.
\end{definition}

\begin{figure}
\begin{center}
\includegraphics[scale=0.8]{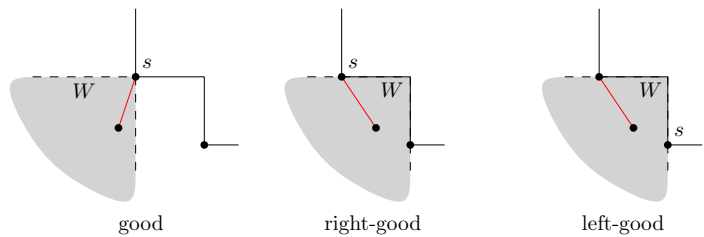}
\end{center}
\caption{The stair-point $s$ is good (resp.\ right-good, left-good) if $W$ contains two points that are connected by an edge.}
\label{fig:goods}
\end{figure}

\begin{definition}
In any step $t$, we say that a point $p$ is {\em good} if any wedge containing $p$ already contains two points connected by an edge, which are thus forced to get different colors (see Figure~\ref{fig:goods}). 
I.e., at any time after $t$, a wedge containing $p$ will contain points of both colors in the final coloring.
A stair-point $p$ is {\em almost-good} if for at least one of its neighbors, $q$, it is true that any wedge containing $p$ and $q$ contains two points connected by an edge of $G_t$.
Additionally, 
if $q$ is the left neighbor of $p$, then we say that $p$ is {\em left-good}, and
if $q$ is the right neighbor of $p$, then we say that $p$ is {\em right-good}. 
\end{definition}

Notice that the good points and the neighbors of the good points are always almost-good.
In fact, good points are also left- and right-good, and a left (resp.\ right) neighbor of a good point is right (resp.\ left) good.

Now we can state the properties we maintain at any time $t$.

\begin{enumerate}[\ \ Property 1.]
\item All above-points are good.
\item All stair-points are almost-good. 
\item All below-points are in different components of $G_t$.
\item $G_t$ is a forest.



\end{enumerate}

For $t=0$, all these properties are trivially true.
Whenever a new point arrives, we execute the following operations (see also Figure~\ref{fig:operations}) repeatedly as long as it is possible, in any order.
This will ensure that the properties remain true.

\begin{figure}
\begin{center}
\includegraphics[width=1\textwidth]{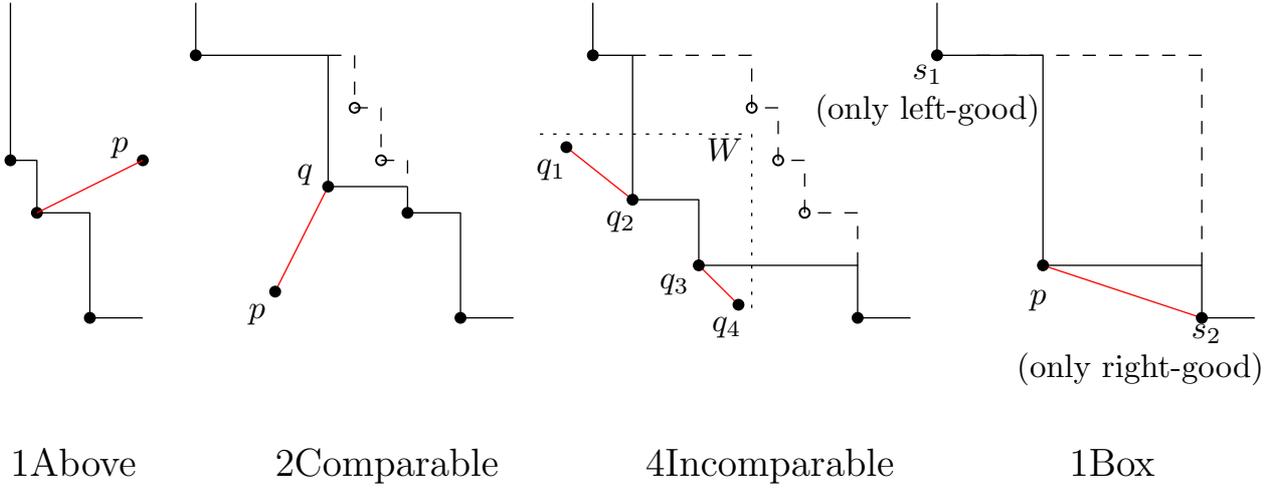}
\end{center}
\caption{The operations maintaining the properties.}
\label{fig:operations}
\end{figure}

\begin{description}
\item[1Above:]
If an above-point $p$ is not good, then we connect $p$ by an edge with a stair-point that is SW from $p$.

\item[2Comparable:] If for some below-points $p,q$ we have that $q$ is NE from $p$, then connect them by an edge and put $q$ on the staircase.

\item[4Incomparable:] Suppose there are no comparable below-points and there is a wedge $W$ that lies entirely below the staircase and contains four incomparable points, $q_1, q_2, q_3,$ and $q_4$, in order of their $x$-coordinates.
Then connect $q_1$ with $q_2$ and also $q_3$ with $q_4$, and put $q_2$ and $q_3$ on the staircase.

\item[1Box:] 
Suppose there are no comparable below-points, and suppose $s_1$ and $s_2$ are two neighboring stair-points, $s_1$ is NW from $s_2$, $s_1$ is left-good but not right-good 
while $s_2$ is right-good but not left-good and $p$ is a point in the rectangle defined by the two opposite vertices $s_1$ and $s_2$.
We connect $p$ and $s_2$, and put $p$ on the staircase.
\end{description}


Now we have to verify that the properties remain true after executing an operation.
Note that this was implicitly proved already in \cite{KP} for {\em 1Above}, {\em 2Comparable} and {\em 4Incomparable}.
First, we make the following observation, which implies that we will have to verify Property 2 only for the new stair-points.

\begin{obs}\label{remainsgood}
If a stair-point is left-good (resp.\ right-good, resp.\ good), then if after an operation this point is still a stair-point, then it remains left good (resp.\ right-good, resp.\ good).
\end{obs}
\begin{proof}
Notice that if a stair-point gets a new left-neighbor, then the new neighbor is either good, or right-good.
Similarly, if a stair-point gets a new right-neighbor, then the new neighbor is either good, or left-good.
\end{proof}

\begin{prop}\label{szorszalhasogatas}
After doing an operation Properties 1-4 remain true.
\end{prop}
\begin{proof}
We check each operation and each property.

\begin{itemize}
\item[] {\em 1Above}: For Property 1, notice that $p$ necessarily has to be the newly arrived point, $p_t$, and it becomes good after the operation.
Property 2 obviously remains true.
For Properties 3 and 4 we use that as $p$ is the newly arrived point, before the operation it is not connected to any other point.

\item[] {\em 2Comparable}: Property 1 remains true as only points NE from $q$ became above-points and thus they all have the edge $pq$ SW from them.
Property 2 remains true as the only new stair-point is $q$, which became good.
Properties 3 and 4 remain true as before the operation $p$ and $q$ were in different tree-components, which are then connected.

\item[] {\em 4Incomparable}: Property 1 remains true as, using that no below-points were comparable, any point that became an above-point has either both $q_1$ and $q_2$, or both $q_3$ and $q_4$ SW from it.
Property 2 remains true as there are only two new stair-points: $q_2$ becomes left-good and $q_3$ becomes right-good.
Properties 3 and 4 remain true as before the operation $q_1, q_2, q_3, q_4$ were all in different tree-components, and after the operation two-two of these are connected in a suitable way.

\item[] {\em 1Box}: Property 1 remains true as, using that no below-points were comparable, there are no new above-points.
Property 2 remains true as $p$ is the only new stair-point and it is right-good.
Properties 3 and 4 remain true if $p$ and $s_2$ are in different tree-components.
This will be proved in Lemma~\ref{remainslost} and Lemma~\ref{remainslower}.
\end{itemize}


\begin{lem}\label{remainslost}
If there is no below-point in the tree-component $T_s$ of a stair-point $s$, then this remains true, i.e., later during the process the component containing $s$ will never contain a below-point.
\end{lem}
\begin{proof}
Suppose that there is no below-point in the tree-component $T_s$ of a stair-point $s$. This trivially remains true when a new point arrives (before doing operations). Then a simple case analysis shows that none of the operations can introduce a below-point to the tree-component $T_s$ of a stair-point $s$:

\begin{itemize}
\item[] {\em 1Above}: Either $T_s$ does not change or only $p$ (an above-point) is added to it.
\item[] {\em 2Comparable}: Only the components of the below-points $p$ and $q$ are joined, as $T_s$ must be a different tree from these two (as it contained no below-point), $T_s$ does not change.
\item[] {\em 4Incomparable}: Only the components of the below-points $q_1,\dots,q_4$ change, as $T_s$ must be a different tree from these (as it contained no below-point), $T_s$ does not change. 
\item[] {\em 1Box}: Either $T_s$ does not change or $s_2\in T_s$ in which case $T_s$ is joined with the tree $T_p$ containing $p$. In the latter case in $T_p$ the only below-point was $p$ (by Property 3), which after the operation becomes a stair-point, so the new tree containing $s$, $T_s'=T_s\cup T_p$ still does not contain a below-point after this operation.\qedhere
\end{itemize}
\end{proof}

\begin{lem}\label{remainslower}
Suppose $s$ is a stair-point and $b$ is a below-point in the tree-component $T_s$ containing $s$.
If $s$ is right-good but not left-good, then $b$ is lower than $s$, that is, $b$ has a smaller $y$-coordinate than $s$.
\end{lem}
\begin{proof}
By Observation~\ref{remainsgood}, we know that if $s$ is right-good but not left-good, then this was also true at the time when $s$ became a stair-point. A simple case analysis of the operations shows that at the time when $s$ becomes a stair-point, the statement holds:

\begin{itemize}
\item[] {\em 1Above}: This is not a possible case as in this case no point becomes a stair-point.
\item[] {\em 2Comparable}: This is not a possible case as necessarily $s$ plays the role of $q$ in the operation, in which case $s$ is good and thus also left-good, contradicting our assumption on $s$.
\item[] {\em 4Incomparable}: Necessarily $s$ plays the role of $q_3$ in the operation, thus $b$ is $q_4$ and so it is lower than $s$, as required.
\item[] {\em 1Box}: Necessarily $s$ plays the role of $p$ in the operation, thus after the operation the below-point in $T_s$ is necessarily the point which was the below point of $T_{s_2}$ before the operation. As $s_2$ was right-good but not left-good, by induction this below point was lower than $s_2$, thus it is also lower than $s=p$, as required.
\end{itemize}
If after some step $T_s$ stops to have a below-point, then by Lemma~\ref{remainslost} this remains true and so there can be no below-point $b$ in $T_s$ as required by the lemma and we are done.
Otherwise, if $T_s$ still has a below-point, then by Property 3 there is exactly one below-point $b$ in $T_s$, it is lower than $s$, and we have to check that after any operation the below-point in $T_s$ remains below $s$. 
The only operation in which the below-point $b$ in $T_s$ could go higher is {\em 4Incomparable} such that $b$ plays the role of $q_2$.
If $b=q_2$ is SW from $s$, then $s$ goes above the staircase, thus stops being a stair-point as required by the lemma and we are done.
If $b=q_2$ is SE from $s$, then the whole wedge $W$ must be lower than $s$, and then the new below-point in $T_s$ becomes $q_1$, also lower than $s$.
This finishes the proof of the lemma and also of Proposition~\ref{szorszalhasogatas}.
\end{proof}\renewcommand{\qedsymbol}{}\end{proof} 
\vspace{-1cm}

\begin{figure}
    \centering
    \subfigure[]{\label{split1}
        \includegraphics[scale=0.8]{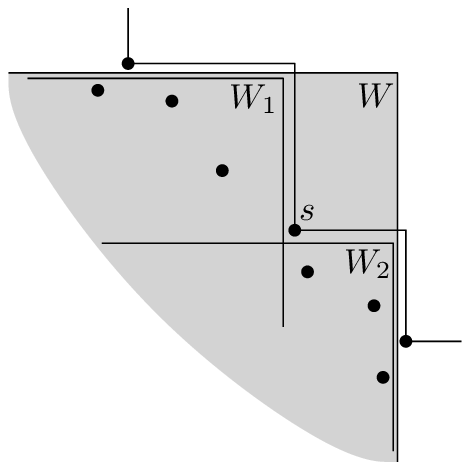}             
				}        
    \hskip 20mm
    \subfigure[]{\label{split2}
        \includegraphics[scale=0.8]{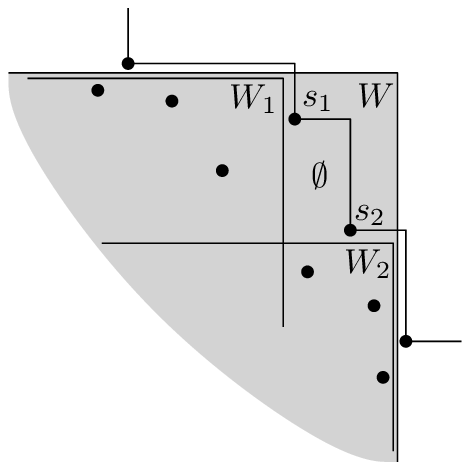}        
				}
   \caption{A monochromatic wedge can contain at most $8$ points.}       
\end{figure}

Now we can finish the proof of the dynamic dual version, and thus also of Theorem~\ref{thm:9}, by showing that taking any (partial) two-coloring of the forest $G_t$ constructed using the above operations, at all times (i.e., for every prefix set $\{p_1,\dots ,p_t\}$ of the point set), any quadrant $W$ containing at least $9$ points contains both colors.
Fix the time after the arrival of the point $p_t$ (and after we repeatedly applied the operations as long as possible). Thus no more operations can be applied, in particular there are no two comparable below-points otherwise we could apply operation {\em 2Comparable}.
If $W$ contains an above-point, it contains both colors as all above-points are good.
If $W$ contains at most one stair-point, $s$, then by ``splitting'' $W$ at $s$ (see Figure~\ref{split1}), we get two quadrants that do not contain any stair-point, but contain all other points that $W$ contains.
One of these two quadrants must contain at least $4$ below-points, thus we could apply operation {\em 4Incomparable}, a contradiction.
If $W$ contains at least $3$ stair-points, then it contains a stair-point $s$ such that both neighbors of $s$ are also in $W$.
As every stair-point is almost-good, $W$ must contain both colors.
Finally, if $W$ contains exactly two (neighboring) stair-points, $s_1$ NW from $s_2$, then the only way for $W$ to be monochromatic is if $s_1$ is left-good but not right-good and $s_2$ is right-good but not left-good.
Therefore, there can be no points in the rectangle formed by $s_1$ and $s_2$, as otherwise we could apply operation {\em 1Box}, a contradiction.
At least one of the two quadrants obtained by ``splitting'' $W$ at $s_1$ and $s_2$ (see Figure ~\ref{split2}), must contain at least $4$ below-points, thus we could apply operation {\em 4Incomparable}, a contradiction.

\section{Indecomposable $4$-fold covering}\label{sec:lower}
Here we construct for any triangle $T$ a finite point set $P$ such that for every two-coloring of $P$ there is a translate of $T$ that contains exactly $4$ points and all of these have the same color.
As the construction is quite hard to describe precisely, we refer to Figure~\ref{fig:haromszogre4nemeleg} and Figure~\ref{fig:haromszogre4nemelegzoom} for the details and give only a more informal description below.
With a simple case analysis, we will show that in any two-coloring, there is a monochromatic triangle with exactly $4$ points.

\begin{figure}
\begin{center}
\includegraphics[width=1\textwidth]{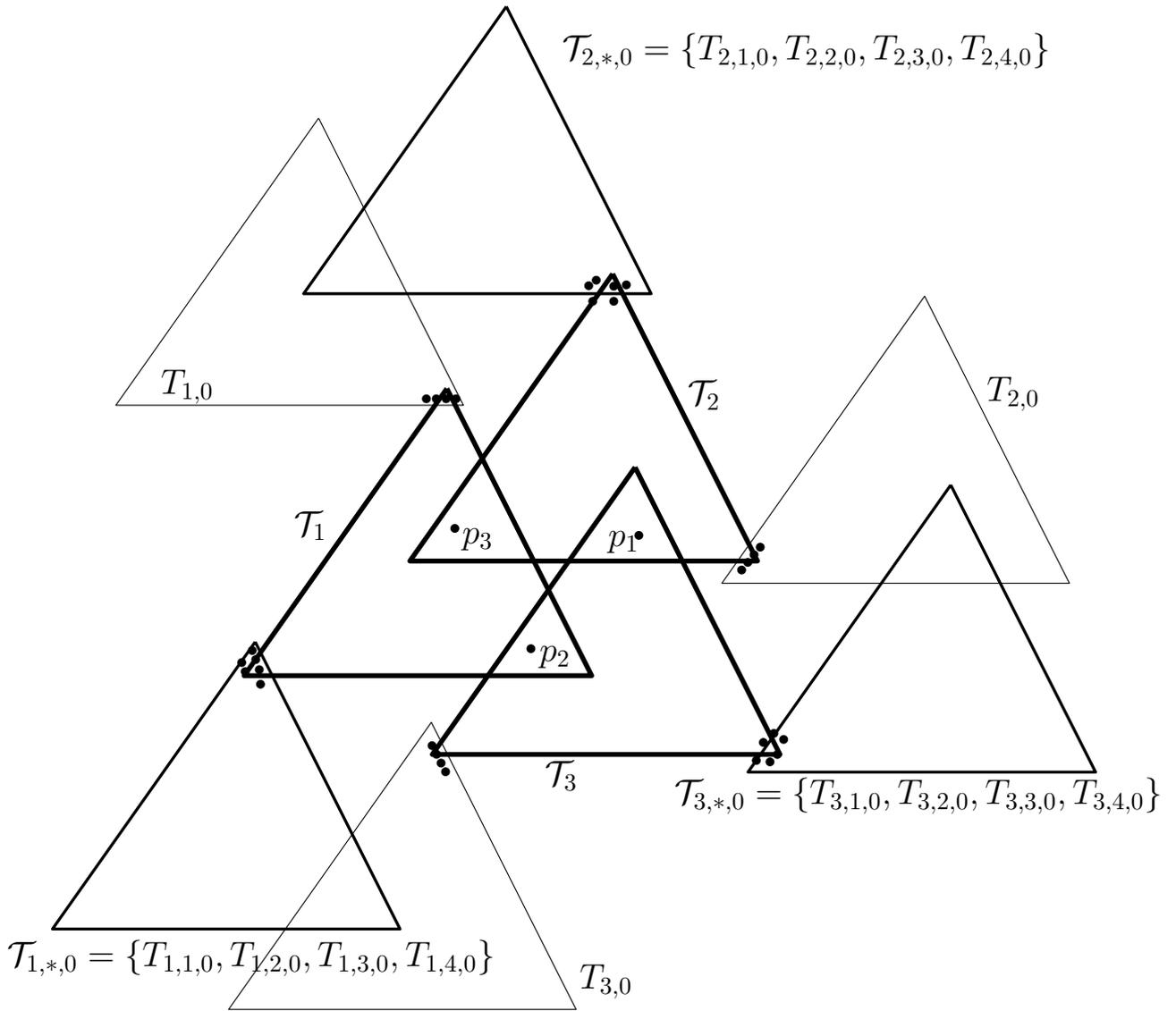}
\end{center}
\caption{The construction in which there is always a monochromatic triangle with $4$ points.}
\label{fig:haromszogre4nemeleg}
\end{figure}

\begin{figure}
\begin{center}
\includegraphics[width=1\textwidth]{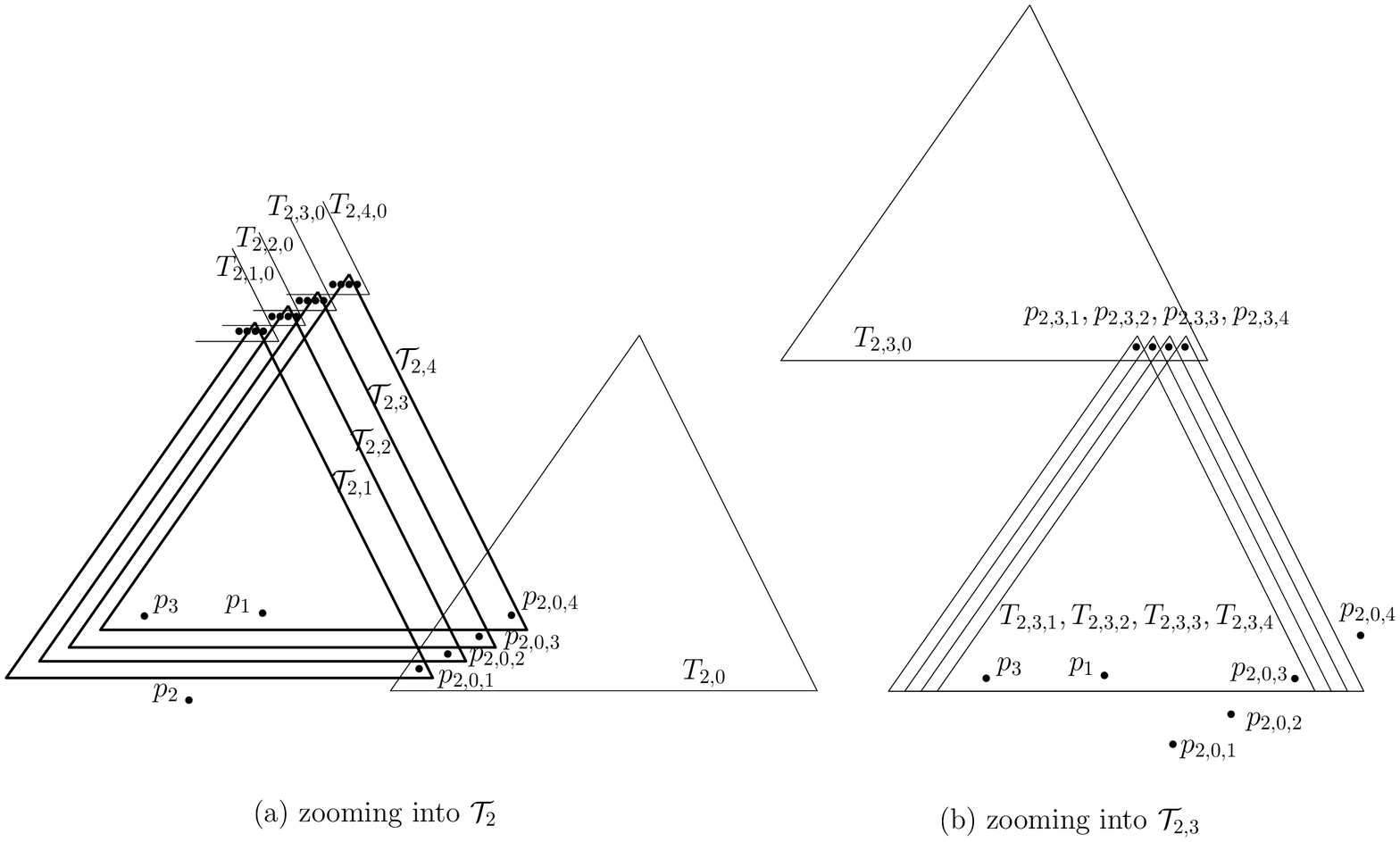}
\end{center}
\caption{Zooming into the construction in two stages.}
\label{fig:haromszogre4nemelegzoom}
\end{figure}

On the top part of the figure is the ``big picture'' that shows what the construction looks like from far.
The thicker triangles denote families of triangles that are very close to each other.
The center part has only three points, $p_1$, $p_2$ and $p_3$.
Two of these, without loss of generality $p_1$ and $p_3$, must receive the same color, say blue.

After this we look more closely at the family $\mathcal T_2$ that consists of the subfamilies $\mathcal T_{2,1}$, $\mathcal T_{2,2}$, $\mathcal T_{2,3}$ and $\mathcal T_{2,4}$, see the bottom-left figure.
Unless the triangle $T_{2,0}$ is monochromatic, at least one of $p_{2,0,1}$, $p_{2,0,2}$, $p_{2,0,3}$ and $p_{2,0,4}$ must be blue.
Without loss of generality we suppose $p_{2,0,3}$ is blue.

After this we look more closely at the family $\mathcal T_{2,3}$ that consists of the triangles $T_{2,3,1}$, $T_{2,3,2}$, $T_{2,3,3}$ and $T_{2,3,4}$, see the bottom-right figure.
Unless the triangle $T_{2,3,0}$ is monochromatic, at least one of $p_{2,3,1}$, $p_{2,3,2}$, $p_{2,3,3}$ and $p_{2,3,4}$ must be blue.
But if $p_{2,i,3}$ is blue, then $T_{2,i,3}$ is monochromatic.
This finishes the proof.

\section{Coloring dynamic hypergraphs defined by intervals}\label{sec:int}
In this section we investigate two-coloring geometric {\em dynamic hypergraphs} defined by intervals on a line. The vertices of a dynamic hypergraph are ordered and they ``appear'' in this order.
Knowing in advance the whole ordered hypergraph, our goal is to color the whole vertex set such that at all times any hyperedge restricted to the vertices that have ``arrived so far'' is non-monochromatic if it contains at least $m$ vertices that have arrived so far.
This model is also called quasi-online in \cite{KNP}.
The exact definitions are as follows.

\begin{definition}
For a hypergraph $\HH(V,{\cal E})$ with an order on its vertices, $V=\{v_1,v_2,\dots,v_n\}$, we define the {\em dynamic closure} of $\cal H$ as the hypergraph on the same vertex set and with hyperedge set $\{E\cap \{v_1,v_2,\dots,v_i\}:E\in{\cal E},1\le i\le |V|\}$. A hypergraph with an order on its vertices is {\em dynamic} if it is its own dynamic closure.
When $\HH$ is a hypergraph family, \dH is the hypergraph family that contains all the dynamic closures of the family $\HH$ (with all orderings of their vertex sets).

A hypergraph is $m$-proper two-colorable if $V$ can be two-colored such that for every $i$ and $E\in{\cal E}$ if $|E|\ge m$, then $E$ contains both colors.
For a family of (ordered) hypergraphs, $\{\HH_i\mid i\in I\}$, we define the {\em midriff} of the family, $m(\{\HH_i\mid i\in I\})$, as the smallest number $m$ such that every (ordered) hypergraph in the family is $m$-proper two-colorable.
\end{definition}

\begin{obs}\label{obs:todynamic}
If for two non-ordered hypergraph families, $\mathcal A$ is a subfamily of $\mathcal B$, then for their dynamic closures, $\mathcal{D\mhyphen A}$ is a subfamily of $\mathcal{D\mhyphen B}$.
\end{obs}

Now we look at the hypergraphs that we have studied in the earlier sections using the above terminology.

\begin{definition}
\end{definition}
\begin{description}
\item[\PQ:] Vertices: a finite set of points from $\R^2$;\\
Hyperedges: subsets of the vertices (points) contained in a quadrant (of the form $(-\infty,x)\times (-\infty,y)$).

\item[\QP:] Vertices: a finite set of quadrants;\\
Hyperedges: subsets of the vertices (quadrants) that contain a point from $\R^2$.

\item[\PO:] Vertices: a finite set of points from $\R^3$;\\
Hyperedges: subsets of the vertices (points) contained in an octant (of the form $(-\infty,x)\times (-\infty,y)\times (-\infty,z)$.

\item[\OP:] Vertices: a finite set of octants;\\
Hyperedges: subsets of the vertices (octants) that contain a point from $\R^3$.
\end{description}

We know that $m_{oct}=\mOP$ by definition and as we noted already in the previous section, in \cite{KP,PPT11} it was shown (not using this terminology) that decomposing octants is equivalent to its dual problem, i.e., \OP$=$\PO and $m_{oct}=\mOP=\mPO$ and that \PO is the same as the (ordered) hypergraph family \dPQ (regarding the hypergraphs in it without the vertex orders). Also, \QP$=$\PQ. 
Summarizing:

\begin{obs}[\cite{KP,PPT11}]
\dPQ equals \PO and therefore\\ $m_{oct}=\mPO=\mPQ$.
\end{obs}

Now we will define the hypergraph families that are the main topic of this section.

\begin{definition}
The set of all intervals\footnote{Note that we are dealing with finitely many objects, so it does not matter if the intervals are closed or open.} on the real line is denoted by $\IIa$.
\end{definition}

\begin{description}
\item[\PI:] Vertices: a finite point set;\\
Hyperedges: subsets of the vertex points contained in an interval $I\in \IIa$.


\item[\IP:] Vertices: a finite set of intervals;\\
Hyperedges: subsets of the vertices (intervals) containing a point $p\in \R$.


\item[\IBI:] Vertices: a finite set of intervals;\\
Hyperedges: subsets of the vertices (intervals) contained in an arbitrary interval $I\in \IIa$.


\item[\ISI:] Vertices: a finite set of intervals;\\
Hyperedges: subsets of the vertices (intervals) containing an arbitrary interval $I\in \IIa$.


\item[\ICI:] Vertices: a finite set of intervals;\\
Hyperedges: subsets of the vertices (intervals) intersecting an arbitrary interval $I\in \IIa$.
\end{description}

Now we study the relations among the above five hypergraph families.
By exchanging points with small enough intervals we get that the family \PI is a subfamily of \IBI, and the family \IP is a subfamily of \ISI, while both \PI and \IP are subfamilies of \ICI.
Observation~\ref{obs:todynamic} implies several inequalities among their midriffs, for example, that $\mPI \le \mIBI$. 

We are only aware of earlier papers studying the first two variants.
It follows from a greedy algorithm that $\mPI=2$ and $\mIP=2$.
It was shown in \cite{wcf} that $\mPI=4$, and later this was generalized for $k$-colors in \cite{A+13}.
It was also shown in \cite{wcf} that $\mIP=3$, and later this proof was simplified in \cite{KNP}.
It is interesting to note that for the \dPI $m$-proper coloring problem 
there is a so-called {\em semi-online} algorithm, that can maintain an appropriate partial $m$-proper coloring of the points arrived so far, while it was shown in \cite{colorful2} that no semi-online algorithm can exist for $m$-proper coloring \dIP.
Here we mainly study the other three hypergraph families.

\begin{figure}[t]
\begin{center}
\includegraphics[scale=0.5]{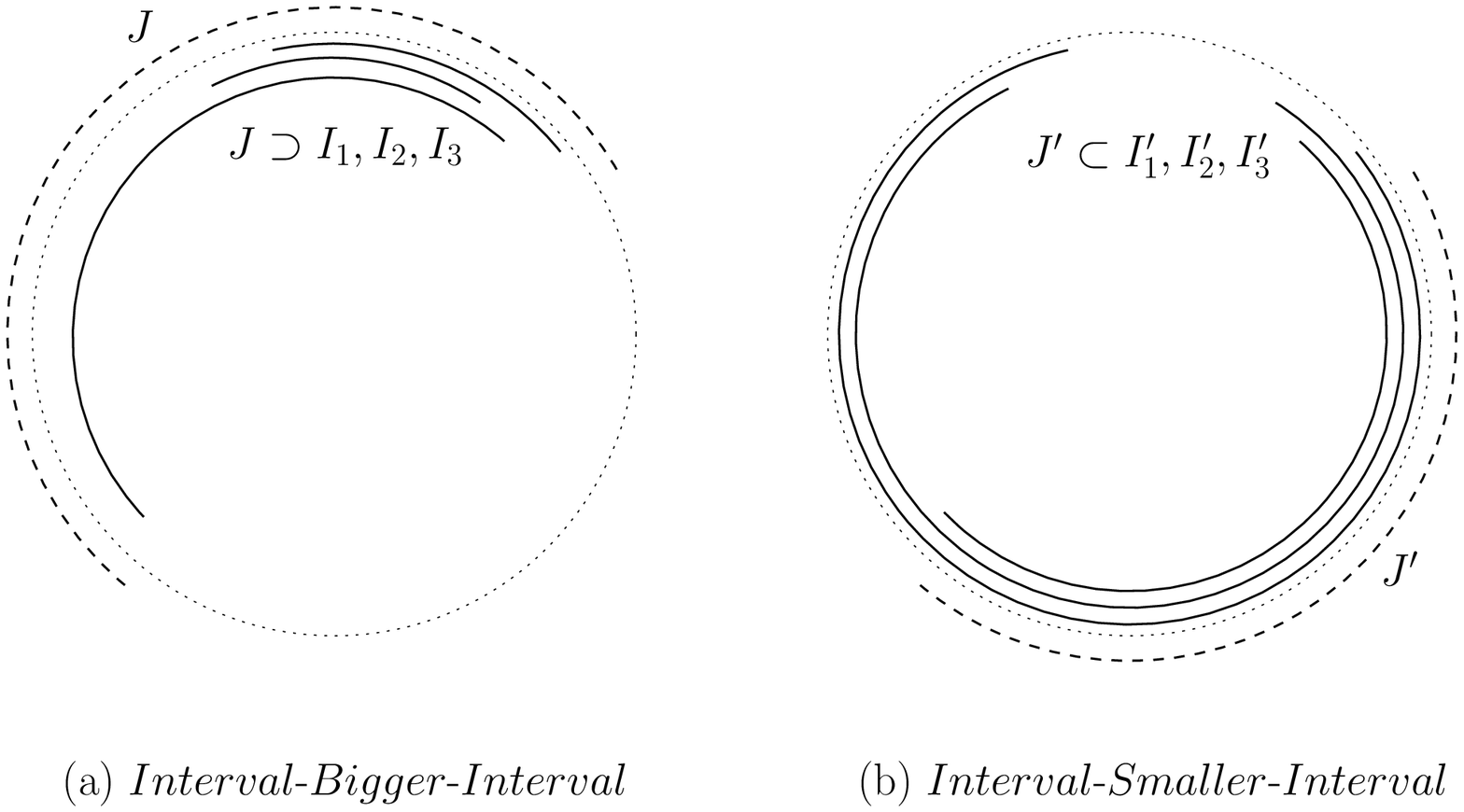}
\end{center}
\caption{\IBI equals \ISI.}
\label{fig:circledual}
\end{figure}

\begin{prop}\label{claim:ibiisi} 
 \IBI equals \ISI, thus \dIBI equals \dISI.
\end{prop}
\begin{proof}
By Observation~\ref{obs:todynamic} it is enough to prove the first statement.
Notice that in both \IBI and \ISI, we can suppose that the left endpoint of any vertex interval is to the left of the right endpoint of any vertex interval, as if we have a right endpoint such that the closest endpoint to its right is a left endpoint, then swapping them does not change the hypergraphs.
Thus, without loss of generality, there is a point that is in all the vertex intervals.
Instead of a line, imagine that the vertex intervals of a hypergraph of \IBI are the arcs of a circle such that none of them contains the bottommost point of the circle and all of them contains the topmost point.\footnote{Without the extra condition regarding the bottommost point, we could define a circular variant of the problem whose parameter $m$ can be at most one larger than \mIBI but we omit discussing this here.}
This is clearly equivalent to the version when the vertex intervals are on the line.
Similarly, we can imagine that the vertex intervals of a hypergraph of \ISI are the arcs of a circle such that none of them contains the topmost point of the circle and all of them contains the bottommost point.
Taking the complement of each arc transforms the families into each other, that is, in a hypergraph $H$ of \IBI an interval $J$ defines the hyperedge $\{I\in H: J\supset I\}$ while in the hypergraph $H'$ of \ISI whose vertices are the complements of the intervals in $H$, the complement $J'$ of $J$ defines the hyperedge $\{I'\in H': J'\subset I\}$. As $J\supset I$ if and only if $J'\subset I'$, $H$ and $H'$ are isomorphic, see Figure~\ref{fig:circledual}.
\end{proof}

\begin{figure}[t]
\begin{center}
\includegraphics[width=1\textwidth]{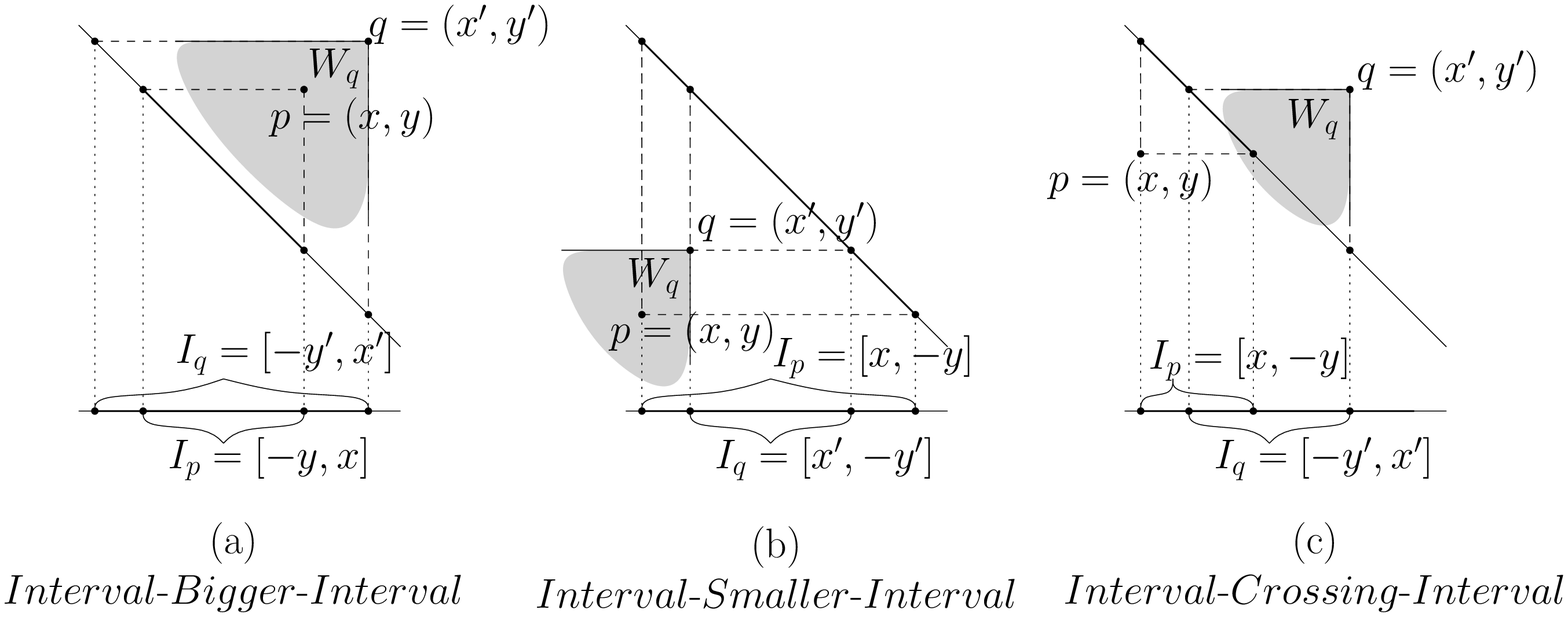}
\end{center}
\caption{Left: \IBI equals \PQ, Center: \ISI equals \PQ, Right: \ICI is a subfamily of \PQ.}
\label{fig:3equi}
\end{figure}

\begin{lem}\label{lem:ibiisi} 
\IBI and \ISI are both equal to \PQ, while
\ICI is a family of subhypergraphs of hypergraphs from the above, and the same holds for the dynamic variants.
\end{lem}
\begin{proof}
By Observation~\ref{obs:todynamic}, it is enough to prove the statements about the non-dynamic families.
For an illustration for the proof, see Figure~\ref{fig:3equi}. 
Recall that a quadrant is a set of the form $(-\infty,x)\times (-\infty,y)$ for some apex $(x,y)$.
We can suppose that all points of the point set are in the North-Eastern halfplane above the line $\ell$ defined by the function $x+y=0$, i.e., $x+y>0$ for every $p=(x,y)$. 
For each point $p=(x,y)$ we define an interval, $I_p=[-y,x]$. 
Quadrants that lie entirely below $\ell$ do not contain points from $P$.
For the quadrants with apex above $\ell$, a quadrant whose apex is at $q$ contains the point $p$ if and only if $I_q$ contains $I_p$.
This shows that the hypergraphs in \PQ and in \IBI are the same.

The equivalence of \ISI and \PQ already follows from Proposition~\ref{claim:ibiisi}, but we could give another proof in the above spirit, by supposing that for all points $p=(x,y)$ we have $x+y<0$, moreover, that for every quadrant intersecting some of the points there is a quadrant containing the same set of points whose apex $q=(x',y')$ has $x'+y'<0$.
Now for each point $p=(x,y)$ we can define the interval $I_p=[x,-y]$ and proceed as before.
Note that this gives another proof for Proposition~\ref{claim:ibiisi}.

Finally, taking a $H$ in \ICI, it is isomorphic to the subhypergraph of some $H'$ in \PQ where in $H$ for all points $p=(x,y)$ we have $x+y<0$ and we take only the hyperedges corresponding to quadrants whose apex $q=(x',y')$ has $x'+y'>0$.
Now for each point $p=(x,y)$ below $\ell$ we define $I_p=[x,-y]$, and for each point $q=(x',y')$ above $\ell$ we define $I_q=[-y',x']$, and proceed as before.
This finishes the proof of the theorem.
\end{proof}

As it was shown in \cite{PT10} that $\mPQ=2$, it follows that also $$\mIBI=\mISI=$$ $$=\mICI=2.$$

Quite surprisingly, we could not find a simple direct proof for the fact that \dICI is a family of subhypergraphs of hypergraphs from the families \dIBI and \dISI.

From Theorems~\ref{thm:9} and~\ref{thm:lower}, and Lemma~\ref{lem:ibiisi} we obtain the following.

\begin{cor} $5 \le \mIBI=m_{oct}\le 9$ and\\ $5 \le \mISI=m_{oct}\le 9.$
\end{cor}

\section{Concluding remarks}

Our concluding diagram can be seen on Figure~\ref{fig:diagram}.
Most importantly, for octants we have $5\le m_{oct}\le 9$ and the same bound holds for the homothets and the translates of triangles.
It seems to be in reach to determine these parameters exactly.
For translates of convex $n$-gons for $n>3$, the parameter $m$ might depend on the shape of the $n$-gon for $n$ fixed, and tends to infinity with $n$ \cite{P13}.
The upper and lower bounds \cite{GV09,PT10} are currently very far, already for a square.
Also, it is not known whether there exists an $m$ (again depending on the convex $n$-gon) such that any finite point set admits a two-coloring such that any homothet of the convex $n$-gon containing at least $m$ points is non-monochromatic, not even for the square.
On the other hand, for any $m$ there is an $m$-fold covering by finitely many homothets of any convex $n$-gon (for $n>3$) of some set that does not decompose to two coverings \cite{K14}.


\subsection*{Acknowledgements}
We are grateful to the anonymous referees for their several useful remarks which greatly improved the presentation of the paper.

\begin{figure}
\begin{center}
\includegraphics[width=1\textwidth]{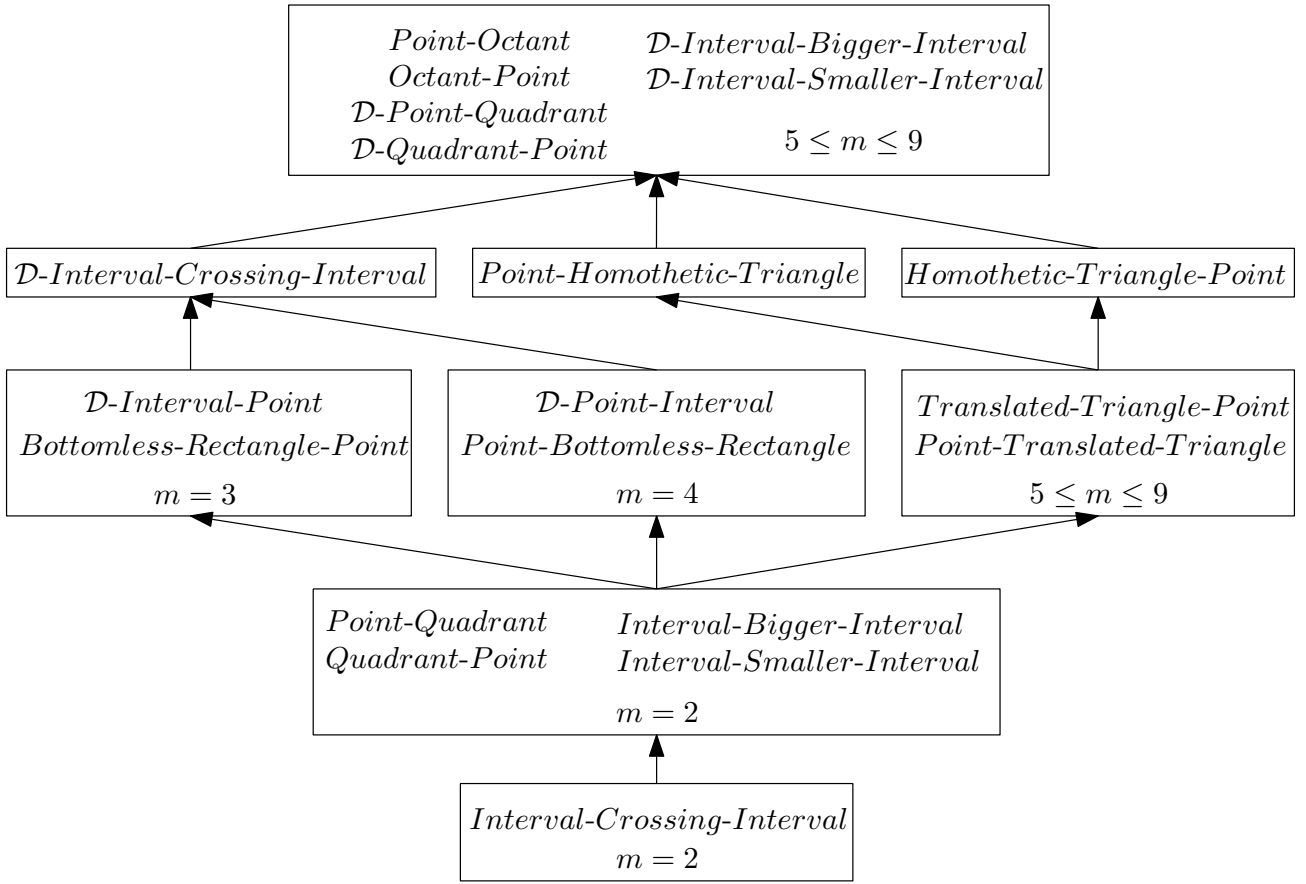}
\end{center}
\caption{Diagram of known results, with edges directed towards more general (dynamic) hypergraphs.
The hypergraph family names follow a similar system as earlier, thus e.g., {\it Homothetic-Triangle-Point} is the family of hypergraphs whose vertex set is a finite set of homothets of a triangle and a subset of these triangles is a hyperedge if and only if there is a point in the plane contained in exactly these homothets. Thus the coloring problem is that we have to color finitely many homothets of a triangle with two colors such that every $m$-fold covered point is contained in both colors.
}
\label{fig:diagram}
\end{figure}

\end{document}